%% file: k-theory-ZR-spaces-v5.tex
\documentclass[a4paper,12pt]{amsart}
\usepackage[margin=2cm]{geometry}

\input{0-header.tex}

\input{0-math-header.tex}
\title{K-Theory of Admissible Zariski-Riemann Spaces}
\author{Christian Dahlhausen}
\address{Mathematisches Institut, Universität Heidelberg, Im Neuenheimer Feld 205, 69120 Heidelberg, Germany}
\email{cdahlhausen@mathi.uni-heidelberg.de}
\urladdr{cdahlhausen.eu}
\thanks{This work was supported by the Swiss National Science Foundation (grant number 184613). The author is supported by Deutsche Forschungsgemeinschaft (DFG) through the Collaborative Research Centre TRR 326 ``Geometry and Arithmetic of Uniformized Structures'' (project number 444845124).}
\hypersetup{
pdftitle={K-Theory of Zariski-Riemann Spaces},
pdfauthor={Christian Dahlhausen},
pdfkeywords={K-Theory, Algebraic Geometry, Homotopy Theory},
pdfmenubar=true
}
\begin{document}
\begin{abstract}
We study relative algebraic K-theory of admissible Zariski-Riemann spaces and prove that it is equivalent to relative G-theory and satisfies homotopy invariance. Moreover, we provide an example of a non-noetherian abelian category whose negative K-theory vanishes.
\end{abstract}
\maketitle
\setcounter{tocdepth}{1}
\tableofcontents

%%%%%%%%%%%%%%%%%%%%%%%%%%%%%%%%%%%%%%%%%%%%%%%%%%%%%%%%%%%%%%%%%%%%%%%
\section{Introduction}
%%%%%%%%%%%%%%%%%%%%%%%%%%%%%%%%%%%%%%%%%%%%%%%%%%%%%%%%%%%%%%%%%%%%%%%

Under the assumption of resolution of singularities, one can obtain for a non-regular scheme $X$ a regular scheme $X'$ which admits a proper, birational morphism $X' \to X$. For many purposes $X'$ behaves similarly to $X$. 
Unfortunately, resolution of singularities is not available at the moment in positive characteristic. 
From the perspective of K-theory, a good workaround for this inconvenience is to work with a Zariski-Riemann type space which is not a scheme anymore, but behaves almost as good as a regular model does. For instance, for a regular noetherian scheme $X$ one has equivalences $\K(\Vect(X))\simeq\K(\Coh(X))$ and $\K(X)\simeq\K(X\times\A^1)$ for Quillen's K-theory. For the \emph{Zariski-Riemann type space} $\skp{X}$ defined to be the limit of all schemes which are projective and birational over a scheme $X$ within the category of locally ringed spaces, Kerz-Strunk-Tamme \cite{kst} established that $\K(\Vect(\skp{X})) \simeq \K(\Coh(\skp{X}))$ and $\K(\Vect(\skp{X})) \simeq \K(\Vect(\skp{X}\times\A^1))$. The purpose of this note is to prove the analogous statement for \emph{admissible Zariski-Riemann spaces} $\skp{X}_U$ which are limits of projective morphisms to a scheme $X$ which are isomorphisms over an open subscheme $U$ (Definition~\ref{U-admissible_ZR-space--defin}). As one only modifies something outside $U$, one has to pass to the relative K-theory. The main results are the following ones, see Theorem~\ref{ZR_K-theory_support_g-theory--thm} and Theorem~\ref{K-ZR-support-homotopy-invariance--thm} as well as their corollaries.

\begin{thm*}
Let $X$ be a reduced, divisorial, and noetherian scheme and let $U$ be a dense open subset of $X$. Denote by $\tilde{Z}$ the complement of $U$ in $\skp{X}_U$. Then 
\begin{enumerate}
\item $\K\bigl(\skp{X}_U \on  \tilde{Z}\bigr) \simeq \K\bigl(\Coh_{\tilde{Z}}(\skp{X}_U)\bigr) \simeq \Ge\bigl(\skp{X}_U \on  \tilde{Z}\bigr)$ and

\item $\K\bigl(\skp{X}_U \on  \tilde{Z}\bigr) \simeq \K\bigl( \skp{X}_U\times\A^1 \on \tilde{Z}\times\A^1\bigr)$.
\end{enumerate} 
%\comm{In particular, if $\tilde{Z}$ is regular, then $\K\bigl(\skp{X}_U \on  \tilde{Z}\bigr) = \K(\tilde{Z})$}.
 Moreover, if $U$ is regular, then 
\begin{enumerate} \setcounter{enumi}{2}
\item $\K(\skp{X}_U) \simeq \Ge(\skp{X}_U)$ and
\item $\K(\skp{X}_U) \simeq \K(\skp{X}_U\times\A^1)$.
\end{enumerate}
where $\Ge(\skp{X}_U) := \K(\Modfp(\skp{X}_U))$ (Definition~\ref{G-theory-ZR--def}).
\end{thm*}

The notion of Zariski-Riemann spaces goes back to Zariski \cite{zariski} who called them ``Riemann manifolds'' and was further studied by Temkin \cite{temkin-relative-zr}. Recently, Kerz-Strunk-Tamme \cite{kst} used them to prove that homotopy algebraic K-theory \cite{weibel-htpy} is the cdh-sheafification of algebraic K-theory, and Elmanto-Hoyois-Iwasa-Kelly applied them to prove a bound on the cdh-cohomological dimension \cite{ehik1}.

\vquad Combining part (i) of the theorem with a result of Kerz about the vanishing of negative relative K-theory, we provide an example of a non-noetherian abelian category whose negative K-theory vanishes (Example~\ref{ZR-not-coherent--ex}). This gives evidence to a conjecture by Schlichting (which was shown to be false at the generality it was stated), see section~\ref{schlichting-conjecture--remark}.

\vquad\noindent\textbf{Notation.}
A scheme is said to be \df{divisorial} iff it admits an ample family of line bundles \cite[2.1.1]{tt90}; such schemes are quasi-compact and quasi-separated.

\vquad\noindent\textbf{Acknowledgements.}
I thank Andrew Kresch for providing a fruitful research environment at the University of Zurich where most of this paper was written as well as Georg Tamme, Matthew Morrow, and Moritz Kerz for helpful conversations. Furthermore, I want to thank Marco Schlichting for pointing out the reference to Hiranouchi-Mochizuki in the proof of Theorem~\ref{ZR_K-theory_support_g-theory--thm}, Quentin Guignard for pointing out a mistake in a previous version as well as Elden Elmanto and Katharina Hübner for some helpful remarks. Last, but not least, I thank the referee for their very detailled reports which pointed out some gaps in my proofs and led to a much improved presentation.

%%%%%%%%%%%%%%%%%%%%%%%%%%%%%%%%%%%%%%%
\section{Admissible Zariski-Riemann spaces}
\label{sec-ZR-spaces}
%%%%%%%%%%%%%%%%%%%%%%%%%%%%%%%%%%%%%%%		

\begin{notation*}
In this section let $X$ be a \emph{reduced} quasi-compact and quasi-separated scheme and let $U$ be a quasi-compact open subscheme of $X$.
\end{notation*}

\begin{defin} \label{U-admissible_ZR-space--defin} \label{ZR-space--defin}
A \df{U-modification} of $X$ is a proper morphism $X'\to X$ of schemes which is an isomorphism over $U$. The category of $U$-modifications of $X$ is defined with morphisms over $X$ and denoted by $\Mdf(X,U)$. The \df{U-admissible Zariski-Riemann space} of $X$ is defined as the limit
	\[
	\skp{X}_U = \lim_{X'\in\Mdf(X,U)} X'
	\]
in the category of locally ringed spaces. This limit exists, its underlying topological space is coherent and sober, and for any $X'\in\Mdf(X,U)$ the projection map $\skp{X}_U\to X'$ is quasi-compact; see Proposition~\ref{limit_spaces--prop} in the appendix.
\end{defin}

\begin{example}
Let $V$ be a valuation ring with fraction field $K$. Then the canonical projection $\skp{\Spec(V)}_{\Spec(K)} \to \Spec(V)$ is an isomorphism as every $\Spec(K)$-modification of $\Spec(V)$ is split according to the valuative criterion for properness.
\end{example}

\begin{lemma} \label{modifications-cofinal-subcats--lemma}
The following are cofinal subcategories of $\Mdf(X,U)$:
\begin{enumerate}
\item The full subcategory spanned by morphisms $X'\to X$ where $X'$ is reduced.
\item If $U$ is schematically dense in $X$: the full subcategory spanned by $U$-admissible blow-ups, \ie{} blow-ups whose centre does not meet $U$. 
\item If $U$ is schematically dense in $X$: the full subcategory spanned by projective morphisms $X'\to X$.
\end{enumerate} 
In particular, under these assumtions the respective limits over these subcategories agree with the $U$-admissible Zariski-Riemann space. 
\end{lemma}
\begin{proof}
(i) is \cite[3.5]{thesis-paper}, (ii) is \cite[3.4]{thesis-paper}, and (iii) follows from (ii).
\end{proof}

\begin{remark} \label{structure-on-tilde-Z--remark}
For the choice of a scheme-structure $Z$ on the closed complement $X\setminus U$, the complement $\skp{X}_U\setminus U$ comes equipped with the structure of a locally ringed space, namely
	\[
	\skp{X}_U \setminus U = \lim_{X'\in\Mdf(X,U)} X'_Z
	\]
where $X'_Z := X'\times_XZ$.
\end{remark}

\begin{remark}
Let $Z$ be a scheme-structure on $X\setminus U$. One can define the \emph{Zariski-Riemann space} $\skp{Z}$ to be the limit over all projective and \emph{birational} morphism $Z'\to Z$ in the category of locally ringed spaces. In general, there does not exist a canonical morphism 
	\[
	\skp{Z} \not\To \skp{X}_U \setminus U
	\]
of locally ringed spaces. Although for $X'\in\Mdf(X,U)$ and a $U$-admissible blow up $\Bl_Y(X')\to X'$ there exists a canonical morphism
	\[
	\Bl_Y(X'_Z) \To \Bl_Y(X'),
	\]
the problem is that the blow-up $\Bl_Y(X'_Z) \to X'_Z$ needs not to be birational in case that $X'_Z \setminus Y$ is not dense in $X'_Z$.
\end{remark}

\subsection*{Comparison to Temkin's relative Riemann-Zariski spaces}
\begin{remark}
Temkin \cite{temkin-relative-zr} introduced the notion of  a \emph{relative Riemann-Zariski\footnote{The author apologises for the switched order of the names ``Zariski'' and ``Riemann'' which tries to be coherent with the different sources.} space} associated with any separated morphism $f\colon Y\to X$ between quasi-compact and quasi-separated schemes. In \emph{loc.~cit.}\ a $Y$-modification is a factorisation
	\[
	Y \To[f_i] X_i \To[g_i] X
	\]
of $f$ where $f_i$ is schematically dominant and $g_i$ is proper. The family of these $Y$-modifications is cofiltered and the \df{relative Riemann-Zariski space} of the morphism $f\colon Y\to X$ is defined as the cofiltered limit
	\[
	\RZ_Y(X) := \lim_i\, X_i
	\]
which is indexed by all $Y$-modifications. 
\end{remark}

\begin{lemma}
If $U$ is a dense in $X$, then there exists a canonical morphism
	\[
	\RZ_U(X) \To \skp{X}_U
	\]
which is an isomorphism.
\end{lemma}
\begin{proof}
If $X$ is reduced, then $U$ is dense in $X$ if and only if it is schematicaly dense in $X$ \stacks{056D}. Hence every $U$-modification in our sense is also a $U$-modification in Temkin's sense with respect to the inclusion $U\inj X$. Given a $U$-modification $U \to[g] X' \to[p] X$ in Temkin's sense, we take a Nagata compactification $U \inj[j] \bar{X}' \to[q] X'$, \ie{} $j$ is an open immersion, $q$ is proper, and $g=q\circ j$ \stacks{0F41}. Then the base change $(p\circ q)\times_XU \colon \bar{X}'\times_XU \to U$ is bijective and split by an open immersion, hence an isomorphism. Thus the $U$-modifications in our sense are cofinal within the $U$-modifications in Temkin's sense.   
\end{proof}

%%%%%%%%%%%%%%%%%%%%%%%%%%%%%%%%%%%%%%%%%%%%%%%%%%%%%%%%%%%%%%%%%%%%%%%
\section{Modules on admissible Zariski-Riemann spaces}
%%%%%%%%%%%%%%%%%%%%%%%%%%%%%%%%%%%%%%%%%%%%%%%%%%%%%%%%%%%%%%%%%%%%%%%

The following results depend on Raynaud-Gruson's \emph{platification par \'eclatement} \cite[5.2.2]{raynaud-gruson}. These results and their proofs are modified versions of results of Kerz-Strunk-Tamme who considered birational and projective schemes over $X$ instead of $U$-modifications,  \cf{} Lemma~6.5 and Proof of Proposition~6.4 in \cite{kst}.

\begin{notation*}
In this section let $X$ be a \emph{reduced} quasi-compact and quasi-separated scheme and let $U$ be a quasi-compact open subscheme of $X$. For any locally ringed space $(Y,\O_Y)$ we denote by $\Modfp(Y)$ the full subcategory of all $\O_Y$-modules spanned by finitely presented objects, \ie{} modules locally of finite presentation.
\end{notation*}

\begin{lemma} \label{ZR--fpmodules-colimit--lemma}
The canonical functor
	\[
	\colim_{X'\in\Mdf(X,U)} \Modfp(X') \To \Modfp(\skp{X}_U).
	\]
is an equivalence (within the 2-category of categories).
\end{lemma}
\begin{proof}
This is a general fact about limits of coherent and sober locally ringes spaces with quasi-compact transition maps, see \cite[ch.~0, 4.2.1--4.2.3]{fuji-kato} and Proposition~\ref{limit_spaces--prop}.
\end{proof}

This equivalence restricts to the full subcategories of vector bundles, \ie{} locally free modules of finite rank.

\begin{lemma} \label{ZR-vect-colimit--lemma}
The canonical functor
	\[
	\colim_{X'\in\Mdf(X,U)} \Vect(X') \To \Vect(\skp{X}_U)
	\]
is an equivalence (within the 2-category of categories).
\end{lemma}
\begin{proof}
Clearly, the pullback of a vector bundle is again a vector bundle. Hence fully faithfulness follows from the corresponding statement for finitely presented modules.
It remains to show that the functor is essentially surjective, so let $F$ be a locally free $\O_{\skp{X}_U}$-module of finite rank. Since the topological space $\skp{X}_U$ is coherent, we may assume that there exists a finite cover $\skp{X}_U = V_1 \cup \ldots\cup V_k$ of quasi-compact open subsets such that $F|_{V_i} \cong \O_{V_i}^{n_i}$ for all $i$ and suitable natural numbers $n_i$. Hence we can argue by induction on $k$ and reduce to the case when $k=2$. There exists a $U$-modification $X_0$ and quasi-compact open subsets $V_1'$ and  $V_2'$ of $X_0$ such that $V_i = p_0\inv(V_i')$ for $i=1,2$ and $V_1\cap V_2 = p_0\inv(V_1'\cap V_2')$ and there exists $\O_{V_i'}$-modules $G_i$ such that $p_0^*G_i \cong F|_{V_i}$ where $p_0 \colon \skp{X}_U \to X_0$ denotes the canonical projection. By general properties of sheaves on limits of locally ringed spaces we have a bijection 
	\[
	\colim_{X_\alpha\in\Mdf(X_0,U)} \Hom_{X_\alpha}(p_{\alpha,0}^*G_1,p_{\alpha,0}^*G_2) \To[\cong] \Hom_{\skp{X}_U}(F|_{V_1},F|_{V_2}).
	\]
where $p_{\alpha,0} \colon X_\alpha\to X_0$ is the transition map \cite[ch. 0, Thm. 4.2.2]{fuji-kato}. Hence there exists an $\alpha$ such that $p_{\alpha,0}^*G_1$ and $p_{\alpha,0}^*G_2$ glue to a locally free sheaf $G$ on $X_\alpha$ which further pulls back to $F$ which shows that the functor in question is essentially surjective.
\end{proof}

\begin{defin} \label{defin-pseudo-coherent}
Let $(Y,\O_Y)$ be a locally ringed space and let $n\geq 0$. An $\O_Y$-module $F$ is said to be \df{pseudo-coherent of Tor-dimension $\mathbf{\leq n}$} iff there exists an exact sequence
	\[
	0 \to E_n\to \ldots \to E_1 \to E_0\to F\to 0
	\]
where $E_n,\ldots, E_1, E_0$ are vector bundles (\ie{} locally free $\O_Y$-modules of finite rank). Denote by $\Mod^{\leq n}(Y)$ and $\Coh^{\leq n}(Y)$ the full subcategories  of $\Mod(Y)$ \resp{} $\Coh(Y)$ spanned by pseudo-coherent $\O_Y$-modules of Tor-dimension $\leq n$.
\end{defin}

\begin{lemma}[{\cf{} \cite[6.5 (i)]{kst}}] \label{leq1-exact--lemma}
Assume that $U$ is dense in $X$. Then:
\begin{enumerate}
\item For every $U$-modifi"-cation $p\colon X' \to X$ the pullback functor preserves modules of Tor-dimension $\leq 1$ and the restricted functor
	\[
	p^* \Colon \Mod^{\mathrm{fp},\leq 1}(X) \To \Mod^{\mathrm{fp},\leq 1}(X')
	\]
is exact.
\item For every morphism $\phi\colon F\to G$ in $\Modfp(Y)$ such that $F, G$, and $\coker(\phi)$ all lie in $\Modfpleq(Y)$ and for every $U$-modification $p\colon X'\to X$ with $X'$ reduced the canonical maps
		\[
		q^*(\ker(\phi)) \to \ker(q^*\phi) \quad\text{ and } \quad q^*(\im(\phi)) \to \im(q^*\phi)
		\]
are isomorphisms.
\end{enumerate}
\end{lemma}
\begin{proof}
(i) If $F$ is a $\O_X$-module of Tor-dimension $\leq 1$, there exists an exact sequence $0 \to E_1 \to[\phi] E_0 \to F \to 0$ where $E_1$ and $E_0$ are $\O_X$-vector bundles. Then the pulled back sequence 
	\[
	0 \To p^*E_1 \To[p^*\phi] p^*E_0 \To p^*F \To 0
	\]
is exact at $p^*E_0$ and $p^*F$. We claim that the map $p^*\phi$ is injective. The $\O_{X'}$-modules $p^*E_1$ and $p^*E_0$ are locally free, say of rank $n$ and $m$, respectively. Let $\eta$ be a generic point of an irreducible component of $X'$. Since $U$ is dense, the map $(p^*\phi)_\eta \colon \O_{X',\eta}^n \to \O_{X',\eta}^m$ is injective since it identifies with the injective map $\phi_{p(\eta)} \colon (E_1)_{p(\eta)} \inj (E_0)_{p(\eta)}$. By Lemma~\ref{modifications-cofinal-subcats--lemma}~(i), we may assume that $X'$ is reduced. For a general point $x\in X'$, the stalk $\O_{X',x}$  embeds into $\prod_\eta \O_{X',\eta}$ where the product is over the generic points of irreducible components of $X'$. Hence the induced map $(p^*\phi)_{x} \colon \O_{X',x}^n \to \O_{X',x}^m$ is injective. Thus $p^*\phi$ is injective at every point of $X'$, hence injective. Now the exactness of $p^*$ follows from the nine lemma.

(ii) This follows directly from (i).
\end{proof}

\begin{defin} \label{support--def}
Let $(Y,\O_Y)$ be a locally ringed space, $Z$ a closed subset of $Y$, and $j\colon (V,\O_V) \inj (Y,\O_Y)$ the inclusion of the open complement. An $\O_Y$-module $F$ has \df{support in} $Z$ iff $j^*F$ vanishes. Denote with a ``$Z$'' in the lower index the full subcategory of those $\O_Y$-modules which have support in $Z$, \eg{} $\Modfp_Z(Y)\subset \Modfp(Y)$.
\end{defin}

\begin{defin} \label{nilu--defin}
Let $C$ be a pointed category. We denote by $\Nilu(C)$ the category whose objects are pairs $(F,\nu)$ where $F$ is an object of $C$ and $\nu\colon F\to F$ is a nilpotent endomorphism, \ie{} there exists a $k$ such that $\nu^k=0$ in the pointed set $\Hom_C(F,F)$. The morphisms are given by those morphisms in $C$ which are compatible with the respective endomorphisms.
\end{defin}

\begin{lemma}[{\cf{} \cite[6.5]{kst}}] \label{platification-support-lemma}
Assume that $X$ is divisorial and that $U$ is dense in $X$. Let $f\colon Y\to X$ be a quasi-projective morphism of finite presentation and let $Z=Y\setminus(Y\times_XU)$. Given a $U$-modification $p\colon X'\to X$, we denote by $q\colon Y'=Y\times_XX'\to Y$ the pullback of $p$ along $f$.
\begin{enumerate}
	\item For every $F\in\Modfp_Z(Y)$ there exists a $U$-modification $p\colon X'\to X$ such that $q^*F$ lies in $\Modfpleq(Y')$.
	\item For every morphism $\phi\colon F\to G$ in $\Modfp_Z(Y)$ there exists a $U$-modification $p\colon X'\to X$ such that $q^*F, q^*G, \ker(q^*\phi), \im(q^*\phi)$, and $\coker(q^*\phi)$ all lie in $\Modfpleq(Y')$.
	\item For every $(F,\nu)\in\Nilu(\Coh_Z(Y))$ there exists a $U$-modification such that there exists a finite resolution
		\[
		0 \to (E_k,\nu_k) \to \ldots \to (E_0,\nu_0) \to (q^*F,q^*\nu) \to 0
		\]
where all $(E_i)$ are locally free of finite rank.		
\end{enumerate}
\end{lemma}
\begin{proof}
(i) Since $X$ has an ample family of line bundles and the map $Y\to X$ is quasi-projective, also $Y$ has an ample family of line bundles \cite[2.1.2.(h)]{tt90}. Hence there exists an exact sequence $E_1 \to[\phi] E_0 \to F \to 0$ where $E_1,E_0$ are $\O_{Y}$-vector bundles. By our assumptions, 
	\[
	\im(\phi)|_U = \ker\bigl(E_0\to F\bigr)|_{U\times_XY} = \ker\bigl(E_0|_U\to 0\bigr) = E_0|_{U\times_XY}
	\]
is flat. By \emph{platification par \'eclatement} \cite[5.2.2]{raynaud-gruson} there exists a $U$-admissible blow-up $p\colon X'\to X$ such that the strict transform $q^\mathrm{st}\im(\phi)$ is flat, \ie{} locally free. Furthermore, $q^\mathrm{st}\im(\phi) = \im(q^*\phi)$, \cf{} Remark~\ref{strict-trafo--rem} below. Hence we obtain an exact sequence
	\[
	0 \To \im(q^*\phi) \To q^*E_0 \To q^*F \To 0,
	\]
hence $q^*F\in\Modfpleq(Y')$ by Lemma~\ref{leq1-exact--lemma} (ii).

(ii) By (i), there is a $U$-modification $p\colon X'\to X$ such that both $q^*F$ and $q^*G$ have Tor-dimension $\leq 1$. Since $q^*$ is a left-adjoint, $q^*(\coker(\phi)) = \coker(q^*\phi)$ so that we may assume that also $\coker(q^*\phi)$ has Tor-dimension $\leq 1$. Hence we can apply Lemma~\ref{leq1-exact--lemma} (ii) and are done.

(iii) We argue by induction on $k>0$ with $\nu^k=0$. The case $k=1$ is follows from (i) and we assume that $k\geq 2$.  By (ii) we find a $U$-modification $p\colon X'\to X$ such that $q^*F, \ker(q^*\nu^{k-1}), \im(q^*\nu^{k-1})$, and $\coker(q^*\nu^{k-1})$ all lie in $\Modfpleq(Y')$. By induction assumption, there exists after further $U$-modification a finite resolution
 	\[
	0\to (E'_k,\nu'_l) \to \ldots \to (E'_0,\nu'_0) \to (\ker(q^*\nu^{k-1}),\nu') \to 0
	\]
where all $E'_i$ are vector bundles and where $\nu'$ is the restriction of $q^*\nu$ to $\ker(q^*\nu^{k-1})$. Similarly, there exists a finite resolution of  $(\im(q^*\nu^{k-1}),\nu'')$ where $\nu''$ is the restriction of $q^*\nu$ to $\im(q^*\nu^{k-1})$. These two finite resolutions now can be patched together to a finite resolution of $q^*F$.
\end{proof}

%%%%%%%%

\begin{lemma} \label{leq1_ZR--lem}
Assume that $X$ is divisorial and that $U$ is dense in $X$. Then the inclusion
	\[
	\Mod^{\mathrm{fp},\leq 1}_{\tilde Z}(\skp{X}_U) \To \Modfp_{\tilde Z}(\skp{X}_U)
	\]
is an equivalence of categories where $\tilde{Z} := \skp{X}_U\setminus U$.
\end{lemma}
\begin{proof}
It suffices to show that the functor is essentially surjective. Given a module $F$ in the target, we find by Lemma~\ref{ZR--fpmodules-colimit--lemma} an $X'\in\Mdf(X,U)$ and module $F'\in\Modfp(X')$ which pulls back to $F$. By desgin, $F'|_U = F|_U =0$, so $F'$ has support away from $U$ and this remains true after pulling back along any $U$-modification. By Lemma~\ref{platification-support-lemma} there exists a $U$-modification $p\colon X''\to X'$ such that $p^*F'$ has Tor-dimension $\leq 1$, thus also $F$ has Tor-dimension $\leq 1$ by Lemma~\ref{leq1-exact--lemma}.
\end{proof}

\begin{remark} \label{strict-trafo--rem}
In the situation of the proofs of Lemma~\ref{platification-support-lemma} and Lemma~\ref{leq1_ZR--lem} we used that $\im(q^*\phi)$ is the strict transform of $\im(\phi)$ along the $U$-admissible blow-up $q\colon X''\to X'$. By definition, $q^\mathrm{st}\im(\phi)$ is the quotient of $q^*\im(\phi)$ by its submodule of sections whose support is contained in the centre of the blow-up $q$ \stacks{080D}. 
Since the surjective map $q^*\im(\phi)\to q^\mathrm{st}\im(\phi)$ factors over the map $\im(q^*\phi) \subset p^*E_0'$, the following commutative diagram has exact rows and exact columns.
\[
\begin{xy} 
\xymatrix{
	& 0 \ar[d] & 0 \ar[d]  \\
	0 \ar[r] & \ker(\sigma) \ar[r] \ar[d] 
	& \im(q^*\phi) \ar[r]^\sigma \ar[d] 
	& q^\mathrm{st}\im(\phi) \ar[d] \ar[r] & 0
	\\
	0 \ar[r] & \ker(\tau) \ar[r] 
	& q^*E_0' \ar[r]^\tau & q^\mathrm{st}E_0' \ar[r] & 0
}
\end{xy}
\]
Since $E_0'$ is a vector bundle, $q^*E_0' = q^\mathrm{st}E_0$ \stacks{080F} which implies the claim.
\end{remark}

If a locally ringed space is cohesive, \ie{} its structure sheaf is coherent (Definition~\ref{cohesive--def}), then a module is coherent if and only if it is finitely presented (Lemma~\ref{cohesive_coherent_fin-pres--lem}). Unfortunately, we do not know whether or not this is also true for the Zariski-Riemann space $\skp{X}_U$. But passing to the complement $\skp{X}_U\setminus U$ we have the following.

\begin{prop} \label{ZR-support-fp=coh--prop}
Assume that $X$ is divisorial and that $U$ is dense in $X$. Let $\tilde{Z}$ be the complement of $U$ in $\skp{X}_U$. An $\O_{\skp{X}_U}$-module with support on $\tilde{Z}$ is coherent if and only if it is finitely presented.
\end{prop}
\begin{proof}
We have to show that every finitely presented $\OXU$-module with support in $\tilde{Z}$ is coherent.  
Let $F$ be a finitely presented $\OXU$-module with support on $\tilde{Z}$. By definition, $F$ is of finite type. Let $V$ be an open subset of $\skp{X}_U$ and let $\phi \colon \O_V^n \to F|_V$ be a morphism. We need to show that $\ker(\phi)$ is of finite type. Since this is a local property and $\skp{X}_U$ is coherent, we may assume that $V$ is quasi-compact. By passing iteratively to another $U$-modification, there exists an $X'\in\Mdf(X,U)$ with canonical projection $p_{X'}\colon \skp{X}_U \to X'$ such that  
\begin{itemize}
\item $F = (p_{X'})^*F_X$ for some $F_X\in\Modfp_{Z'}(X')$ (Proposition~\ref{limit_spaces--prop} (iv)),
	\item $F_X$ has Tor-dimension $\leq 1$ (Lemma~\ref{leq1_ZR--lem}),
	\item $V=(p_{X'})\inv(V')$ for some open subset $V'$ of $X'$ (Proposition~\ref{limit-top-spaces--prop}~(i)),
	\item $\phi$ is induced by a morphism $\phi' \colon \O_{V'}^n \to F_X|_{V'}$ (Proposition~\ref{limit_spaces--prop}~(v)), and
	\item $\coker(\phi)$ has Tor-dimension $\leq 1$ (Lemma~\ref{leq1_ZR--lem}).
\end{itemize}	
Since $\ker(\phi')$ is of finite type we may assume that there exists a surjection $\O_{V'}^m \surj \ker(\phi')$ for some $m\in\N$ (otherwise we have to shrink $V'$). By Lemma~\ref{leq1-exact--lemma} $(p_X)^*\bigl( \ker(\phi') \bigr) = \ker(\phi)$, hence it is of finite type.
\end{proof}

\begin{thm} \label{leq1-resumee--thm}
Assume that $X$ is divisorial and that $U$ is dense in $X$. Let $\tilde{Z}$ be the complement of $U$ in $\skp{X}_U$. Then the canonical functors
	\carre{\Coh_{\tilde Z}^{\leq 1}(\skp{X}_U)}{\Coh_{\tilde Z}(\skp{X}_U)}{\Mod^{\mathrm{fp}, \leq 1}_{\tilde{Z}}(\skp{X}_U)}{\Modfp_{\tilde{Z}}(\skp{X}_U)}
are equivalences. In particular, the canonical functor 
		\[
		\colim_{X'\in\Mdf(X,U)} \,\Coh_{Z'}(X') \To \Coh_{\tilde Z}(\skp{X}_U)
		\]
is an equivalence of categories  where $Z'=X'\setminus U$ and the colimit is taken in the 2-category of categories.
\end{thm}
\begin{proof}
The lower horizontal functor is an equivalence by Lemma~\ref{leq1_ZR--lem} and the right vertical functor is an equivalence by Lemma~\ref{ZR-support-fp=coh--prop}. The other two functors are equivalence since the square is a pullback square. The equivalence from Proposition~\ref{limit_spaces--prop} restricts to an equivalence
	\[
		\colim_{X'\in\Mdf(X,U)} \,\Modfp_{Z'}(X') \To \Modfp_{\tilde Z}(\skp{X}_U)
	\]
which yields the desired statement.
\end{proof}

\begin{cor} \label{ZR-special-fibre-cohesive--cor}
Assume that $X$ is divisorial and that $U$ is dense in $X$. Let $\tilde{Z}$ be the complement of $U$ in $\skp{X}_U$ with induced structure of a locally ringed space (\cf{} Remark~\ref{structure-on-tilde-Z--remark}). Then $\O_{\tilde{Z}}$ is a coherent $\O_{\tilde{Z}}$-module, \ie{} the locally ringed space $\tilde{Z}$ is cohesive.
\end{cor}
\begin{proof}
By Lemma~\ref{coherence-after-pushforward--lem} below it is sufficient to show that $i_*\O_{\tilde{Z}}$ is a coherent $\O_{\skp{X}_U}$-module where $i\colon\tilde{Z} \inj \skp{X}_U$ denotes the inclusion map. This follows from Proposition~\ref{ZR-support-fp=coh--prop} since $i_*\O_{\tilde{Z}}$ is finitely presented by Lemma~\ref{fin-pres-back-and-fro--lem} below.
\end{proof}

\begin{lemma} \label{coherence-after-pushforward--lem}
Let $i\colon Z\inj Y$ be a closed immersion of locally ringed spaces and let $F\in\Mod(Z)$ of finte type. If $i_*F\in\Coh(Y)$, then $F\in\Coh(Z)$.
\end{lemma}
\begin{proof}
Let $\phi\colon\O_W^n\to F|_W$ be a morphism for an open subset $W\subset Z$. Choose an open subset $V \subset Y$ such that $W=V\cap Z$. Then the kernel of the morphism $\psi \colon \O_V^n \to i_*\O_W^n \to[i_*\phi]i_*(F|_W)$ is of finite type if $i_*F$ is coherent so that we get an epimorphism $\O_V^m\to\ker(\psi)$. Then the induced morphism $\O_W^m \to[i_*i^*\psi] i^*\ker(\psi) \to \ker(\phi)$ witnesses $\ker(\phi)$ to be of finite type.
\end{proof}

\begin{lemma} \label{fin-pres-back-and-fro--lem}
Assume that $X$ is noetherian and that $U$ is dense in $X$. Let $i \colon \tilde{Z} \inj \skp{X}_U$ be the inclusion of the complement with induced structure of a locally ringed space (\cf{} Remark~\ref{structure-on-tilde-Z--remark}). Then the functor $i_* \colon \Mod(\tilde{Z}) \to \Mod(\skp{X}_U)$ restricts to a functor $i_* \colon \Modfp(\tilde{Z}) \to \Modfp(\skp{X}_U)$.
\end{lemma}
\begin{proof}
Let $F=p^*F'$ with $F'\in\Modfp(X')$ for some $X'\in\Mdf(X,U)$ with projection $p\colon\skp{X}_U\to X'$. Then $i^*F\cong q^*k^*F'$ for the induced morphisms $q\colon \tilde{Z}\to Z':=Z\times_XX'$ and $k\colon Z'\inj X'$. Note that the functor $k_*$ preserves finitely presented objects since $X'$ is noetherian. We claim that $i_*i^*F \cong p^*k_*k^*F'$ which implies the assertion of the lemma. The unit of the adjunction $q^*\dashv q_*$ induces a morphism $k_*k^*F' \to k_*q_*q^*k^*F' = p_*i_*q^*k^*F'$ which induces a morphism $\phi \colon p^*k_*k^*F' \to i_*q^*k^*F'=i_*i^*F$ by the adjunction $p^*\dashv p_*$. Applying the functor $i^*$ to $\phi$ yields the identity map on $q^*k^*F'$, hence $\phi$ is an isomorphsim since both its source and its target are supported on $\tilde{Z}$.
\end{proof}

\begin{reminder}
Recall that an \df{exact category} is an additive category $\Acal$ together with a class of \df{conflations} 
	\[
	A \infl B \surj C
	\]
satisfying certain axioms, \cf{} \cite[1.1]{schlichting04}. The morphism $A\infl B$ appearing in conflations are called \df{inflations} and the morphisms $B\surj C$ are called \df{deflations}. Every abelian category is an exact category whose conflations are the the short exact sequences; hence the inflations are the monomorphisms and the deflations are the epimorphisms. 
\end{reminder}

\begin{defin} \label{left-s-filtering--defin}
Let $\Acal$ be an exact subcategory of an exact category $\Bcal$. We say that $\Acal\subset\Bcal$ is ...
\begin{enumerate}
\item a \df{Serre subcategory} iff for every conflation $X' \rightarrowtail X \surj X''$ in $\Bcal$ we have that $X\in\Acal$ if and only if both $X'$ and $X''$ lie in $\Acal$.
\item \df{right-filtering} iff it is a Serre subcategory and if for every morphism $f\colon B \to A$ with $B\in\Bcal$ and $A\in\Acal$ there exists an object $A'\in\Acal$ such that $f$ can be factored as a composition $B \surj A' \to A$ where the morphism $B \surj A'$ is a deflation.
\item \df{right-s-filtering} iff it is right-filtering and if for every inflation $A\infl B$ in $\Bcal$ with $A\in\Acal$ there exists a deflation $B \surj A'$ with $A'\in\Acal$ such that the composition $A \infl B \surj A'$ is an inflation.
\end{enumerate}
\end{defin}

\begin{lemma} \label{modules-support-left-s-filtering-lemma}
Let $Y$ be a locally ringed space, $j\colon V\inj Y$ be the inclusion of an open subspace, and $i\colon Z\inj Y$ be the inclusion of its closed complement. For a full exact subcategory $\Mcal$ of $\Mod(Y)$ denote by $\Mcal_Z$ the full subcategory of $\Mcal$ which is spanned by modules with support on $Z$. Assume that $Z$ has the structure of a locally ringed space such that for every $M\in\Mcal$ the object $i_*i^*M$ lies in $\Mcal$ (and hence in $\Mcal_Z$). Then $\Mcal_Z$ is a right-s-filtering subcategory of $\Mcal$.
\end{lemma}
\begin{proof}
The inclusion $\Mcal_Z \inj \Mcal$ is the kernel of the exact restriction functor $j^*\colon \Mcal \to \Mod(U)$, hence it is closed under subobjects, quotients, and extension, thus a Serre subcategory.
Given a morphism $f\colon B\to A$ in $\Mcal$ with $A\in\Mcal_Z$, we have a factorisation of $f$ as 
	\[
	B \surj[\epsilon_B] i_*i^*B \to[i_*i^*f] i_*i^*A \to[(\epsilon_A)\inv] A.
	\]
and the composition $(\epsilon_A)\inv\circ(i_*i^*f)$ is a deflation. Thus $\Mcal_Z$ is right-filtering.
Now let $A\infl B$ be an inflation (\ie{} a monomorphism whose cokernel lies in $\Mcal$) with $A\in\Mcal_Z$. Then the unit $B\surj i_*i^*B$ is a deflation and the composition $A \infl B \to[\eta_B] i_*i^*B$ is an inflation as it equals the composition $A \to[\eta_A] i_*i^*A \infl i_*i^*B$.
\end{proof}

\begin{prop} \label{ZR_localisation--prop}
Assume that $X$ is noetherian. Denote by $\tilde{Z}$ the closed complement of $U$ in $\skp{X}_U$. Then $\Modfp_{\tilde Z}(\skp{X}_U)$ is a right-s-filtering subcategory of the exact category $\Modfp(\skp{X}_U)$ and the inclusion $j\colon U \inj \skp{X}_U$ induces an equivalence of categories
	\[
	j^* \colon\, \Modfp(\skp{X}_U)/\Modfp_{\tilde Z}(\skp{X}_U) \To[\cong] \Modfp(U).
	\]
%If we additionally assume that $X$ is divisorial and that $U$ is dense in $X$, then we have an ``exact sequence'' of exact categories
%	\begin{align*} 
%	\Coh_{\tilde Z}(\skp{X}_U) \To \Modfp(\skp{X}_U) \To \Coh(U)
%	\end{align*}
%(where the first and the last one are abelian), \ie{} the composition functor is zero, the first functor is fully faithful such that the quotient category exists, and the last one 
\end{prop}
\begin{proof} 
Let $i\colon \tilde{Z}\inj\skp{X}_U$ be the inclusion map. For $F\in\Modfp(\skp{X}_U)$ we have that $i_*i^*F\in\Modfp(\skp{X}_U)$ by Lemma~\ref{fin-pres-back-and-fro--lem}. Applying Lemma~\ref{modules-support-left-s-filtering-lemma} with $Y:=\skp{X}_U$, $V:=U$, $Z:=\tilde{Z}$, and $\Mcal := \Modfp(\skp{X}_U)$ we obtain that $\Modfp_{\tilde{Z}}(\skp{X}_U)$ is a right-s-filtering subcategory of $\Modfp(\skp{X}_U)$. Hence the quotient category 
	\[
	\Modfp(\skp{X}_U)/\Modfp_{\tilde Z}(\skp{X}_U)
	\]
is defined \cite[1.12, 1.14]{schlichting04}. By definition, the restriction $j^* \colon \Modfp(\skp{X}_U)\to \Modfp(U)$ factors through the canonical functor $\Modfp(\skp{X}_U) \to\Modfp(\skp{X}_U) / \Modfp_{\tilde Z}(\skp{X}_U)$. We will show that the induced functor
	\[
	j^* \Colon \Modfp(\skp{X}_U)/\Modfp_{\tilde{Z}}(\skp{X}_U) \To \Modfp(U)
	\]
is fully faithful and essentially surjective, following the lines of the classical proof for noetherian schemes, \cf{} \cite[2.3.8]{schlichting11}.

For essential surjectivity let $F\in\Modfp(U)=\Coh(U)$. The inclusion $j_X \colon U \inj X$ induces an equivalence of categories
	\[
	(j_X)^* \Colon \Coh(X)/\Coh_Z(X) \To \Coh(U) 
	\]
where $Z:=X\setminus U$ \cite[2.3.8]{schlichting11}. Thus there exists an $F_X\in\Coh(X) = \Modfp(X)$ such that $(j_X)^*F_X\cong F$.
Since $j_X$ factors as $p_X\circ j$, it follows that
	\[
	F \cong j_X^*F_X \cong (p_X\circ j)^*F_X \cong j^*(p_X^*F_X),
	\]
\ie{} $F$ comes from a finitely presented module $p_X^*F_X \in \Mod\fp(\skp{X}_U)$.

For fullness let $\phi \colon j^*F\to j^*G$ be a morphism of $\O_U$-modules with $F$ and $G$ in $\Mod\fp(\skp{X}_U)$. The quotient of $\Mod\fp(\skp{X}_U)$ by $\Modfp_{\tilde Z}(\skp{X}_U)$ is the localisation of $\Mod\fp(\skp{X}_U)$ along those morphisms which are sent to isomorphisms by $j^*$. Consider the pullback diagram
	\[\begin{xy}  \tag{$\Delta$}  \xymatrix{ 
	H \ar[rr] \ar[d]_\alpha && G \ar[d]^\beta \\
	F \ar[r] & j_*j^*F \ar[r]^{j_*(\phi)} & j_*j^*G
	}\end{xy}\]
in $\Mod(\skp{X}_U)$. Since $j^*$ is exact, the square $j^*(\Delta)$ is also a pullback; hence $j^*(\alpha)$ is an isomorphism (as $j^*(\beta)$ is one). Thus the span $F\ot H\to G$ represents a morphism $\psi$ in the quotient category such that $\phi = j^*(\phi)$, hence $j^*$ is full.

For faithfullness let $F'\ot[\eta]F\to[\phi] G$ be a span representing a morphism in the quotient category $\Modfp(\skp{X}_U)/\Modfp_{\tilde{Z}}(\skp{X}_U)$ which becomes zero when restricted to $U$. Let $\phi'\colon F'\to G'$ be a morphism in $\Modfp(X')$ for some $X'\in\Mdf(X,U)$ with $\phi=p^*(\phi')$ for the projection $p\colon\skp{X}_U\to X'$. Then $\phi'|_U=0$ by design so that $\phi'=0$ in the quotient category $\Modfp(X')/\Modfp_{Z'}(X')$, where $Z'=X'\setminus U$, by the classical case for noetherian schemes. Hence $\phi$ and also the morphism represented by $F'\ot[\eta]F\to[\phi] G$ are zero in $\Modfp(\skp{X}_U)/\Modfp_{\tilde{Z}}(\skp{X}_U)$.

%The last assertion follows from the equality $\Coh(U)=\Modfp(U)$ and Lemma~\ref{ZR-support-fp=coh--prop}.
\end{proof}

%%%%%%%%%%%%%%%%%%%%%%%%%%%%%%%%%%%%%%%
	\section{K-theory of admissible Zariski-Riemann spaces}
	\label{sec-k-theory-zr}
%%%%%%%%%%%%%%%%%%%%%%%%%%%%%%%%%%%%%%%	

In this section we prove for the K-theory of the Zariski-Riemann space a comparison with G-theory (Theorem~\ref{ZR_K-theory_support_g-theory--thm}, Corollary~\ref{ZR-K=G-U-regular--cor}) and homotopy invariance (Theorem~\ref{K-ZR-support-homotopy-invariance--thm}, Corollary~\ref{K-ZR-U-regular-homotopy-invariance--cor}). Finally, we provide an example of a non-noetherian abelian category whose negative K-theory vanishes (Example~\ref{ZR-not-coherent--ex}).

\begin{notation*}
In this section let $X$ be a \emph{reduced} quasi-compact and quasi-separated scheme and let $U$ be a quasi-compact open subscheme of $X$. Denote by $\tilde{Z}$ the complement of $U$ within the $U$-admissible Zariski-Riemann space $\skp{X}_U$ (Definition~\ref{U-admissible_ZR-space--defin}). We equip the closed complement $X\setminus U$ with the reduced scheme structure so that $\tilde{Z}$ has the structure of a locally ringed space (Remark~\ref{structure-on-tilde-Z--remark}).
\end{notation*}

\begin{defin} \label{K-theory-ZR--def}
For an exact category $\Ecal$ we denote by $\K(\Ecal)$ its \emph{non-connective} K-theory spectrum as defined by Schlichting \cite{schlichting04}.\footnote{Beware that in \emph{loc.~cit.}\ this object is denoted by ``$\mathbb{K}$''.} If $X$ is a divisorial scheme, then $\K(\Vect(X))$ is equivalent to the non-connective algebraic K-theory spectrum of $X$ \`{a} la Thomason-Trobaugh \cite{tt90}. 

We define the \df{K-theory of the Zariski-Riemann space} as
	\[
	\K(\skp{X}_U) := \K\bigl( \Vect(\skp{X}_U) \bigr) 
	\]
and the \df{K-theory with support} as
	\[
	\K(\skp{X}_U\on\tilde{Z}) := \fib\bigl( \K(\skp{X}_U) \To \K(U) \bigr).
	\]
\end{defin}

\begin{lemma} \label{K-ZR-colimit--lemma}
The canonical maps 
	\begin{align*}
	\colim_{X'\in\Mdf(X,U)}\K(X') &\To \K(\skp{X}_U) \\
	\colim_{X'\in\Mdf(X,U)}\K(X'\on Z') &\To \K(\skp{X}_U\on\tilde{Z})
	\end{align*}
are equivalences where $Z' := Z\times_XX'$.	 
\end{lemma}
\begin{proof}
The first equivalence follows from Lemma~\ref{ZR-vect-colimit--lemma} and the fact that K-theory commutes with filtered colimits of exact functors; in non-negative degrees this is due to Quillen \cite[p.~20]{quillen73} and in negative degrees due to Schlichting \cite[\S 7, Cor.~5]{schlichting06}. The second one follows since filtered colimits commute with finite limits.
\end{proof}

\begin{thm}[Kerz] \label{ZR_negative_K_support_vanishing--thm}
If $X$ is reduced and divisorial, then 
	\[
	\K_i(\skp{X}_U\on\tilde{Z})=0 \quad\text{ for } i<0.
	\]
Furthermore, if $U$ is additionally regular, then we have $\K_i\bigl(\skp{X}_U\bigr)\cong 0$ for every $i<0$.
\end{thm}
\begin{proof}
The first statement follows from Lemma~\ref{K-ZR-colimit--lemma} and a result of Kerz saying that every negative relative K-theory class vanishes after some admissible blow-up \cite[Prop.~7]{kerz-icm}. The second statement is an immediate consequence using that negative K-theory of regular schemes vanishes.
\end{proof}

\subsection{Comparison with G-theory}

For divisorial noetherian schemes, G-theory is the K-theory of the abelian category of coherent modules (which is the same as the category of finitely presented modules). Over an arbitrary scheme (or locally ringed space) a finitely generated module may not be coherent. Since we want the category of vector bundles to be included, we work with finitely presented modules for defining G-theory for admissible Zariski-Riemann spaces.

\begin{defin} \label{G-theory-ZR--def}
We define the \df{G-theory of the Zariski-Riemann space} as
	\[
	\Ge(\skp{X}_U) := \K\bigl( \Modfp(\skp{X}_U) \bigr).
	\]
and the \df{G-theory with support} as
	\[
	\Ge(\skp{X}_U\on\tilde{Z}) := \fib\bigl( \Ge(\skp{X}_U) \To \Ge(U) \bigr).
	\]
\end{defin}

\begin{prop} \label{ZR-G-theory-fibre-sequence--prop}
Assume that $X$ is divisorial and noetherian and that $U$ is dense in $X$. Then there is a fibre sequence
	\[
	\K\bigl( \Coh_{\tilde{Z}}(\skp{X}_U) \bigr) \To \Ge(\skp{X}_U) \To \Ge(U).
	\]
In other words, the canonical map
	\[
	\K\bigl( \Coh_{\tilde{Z}}(\skp{X}_U) \bigr) \To \Ge(\skp{X}_U\on\tilde{Z})
	\]
is an equivalence.
\end{prop}
\begin{proof}
By Propostion~\ref{ZR_localisation--prop} the category $\Modfp_{\tilde Z}(\skp{X}_U)$ is a right-s-filtering subcategory of the exact category $\Modfp(\skp{X}_U)$ and the quotient is equivalent to $\Modfp(U)$. Applying Schlichting's localisation theorem for additive categories \cite[2.10]{schlichting04} yields the desired fibre sequence.
\end{proof}

\begin{lemma} \label{connective-k-theory-special-fibre--lem}
The canonical map
	\[
	\K_{\geq 0}\bigl( \Coh(\tilde{Z}) \bigr) \To \K_{\geq 0}\bigl( \Coh_{\tilde{Z}}(\skp{X}_U) \bigr)
	\]
between the connective covers of K-theory spectra is an equivalence.
\end{lemma}
\begin{proof}
We claim that that the fully faithful pushforward functor $\Coh(\tilde{Z}) \inj \Coh_{\tilde{Z}}(\skp{X}_U)$ satisfies the conditions of the D\'{e}vissage Theorem \cite[V.4.1]{weibel}.
Let $F \in \Coh_{\tilde{Z}}(\skp{X}_U)$. Using iteratively Lemma~\ref{leq1-exact--lemma} and Lemma~\ref{platification-support-lemma}, there exist $X' \in \Mdf(X,U)$ and a filtration
	\[
	F'_0 \supset F'_1 \supset \ldots \supset F'_n = 0
	\]
with $F_i' \in \Coh^{\leq 1}_{Z'}(X')$ such that $F\cong p^*F_0'$ and such that the quotients $F_i'/F_{i-1}'$ lie in the essential image of the functor $\Coh(Z') \to \Coh_{Z'}(X')$ where $Z' = Z\times_XX'$. Then pulling back to $\skp{X}_U$ yields a filtration of $F$ such that the successive quotients lie in $\Coh(\tilde{Z})$ as desired.
\end{proof}
\begin{remark}
Combining Theorem~\ref{ZR_negative_K_support_vanishing--thm} above with Theorem~\ref{ZR_K-theory_support_g-theory--thm} below we see that the connective cover 
$\K_{\geq 0}\bigl( \Coh_{\tilde{Z}}(\skp{X}_U) \bigr) \To \K\bigl( \Coh_{\tilde{Z}}(\skp{X}_U) \bigr)$
is an equivalence. On the hand, the author does not know whether or not the connective cover 
$\K_{\geq 0}\bigl( \Coh(\tilde{Z}) \bigr) \To \K\bigl( \Coh(\tilde{Z}) \bigr)$
is an equivalence as well.
\end{remark}

\begin{thm} \label{ZR_K-theory_support_g-theory--thm}
Assume that $X$ is divisorial and noetherian and that $U$ is dense in $X$. Then we have an equivalences of spectra
	\[
	\K\bigl(\skp{X}_U \on \tilde{Z}\bigr) \simeq \K\bigl( \Coh_{\tilde{Z}}(\skp{X}_U) \bigr)
	\simeq \Ge\bigl(\skp{X}_U \on \tilde{Z}\bigr).
	\]
\end{thm} 
\begin{proof} 
Every $U$-modification of $X$ has an ample family of line bundles since it is quasi-projective over a scheme admitting such a family \cite[2.1.2~(h)]{tt90}. Moreover, those $X'\in\Mdf(X,U)$ such that $X'\setminus U \inj X'$ is a regular closed immersion form a cofinal subsystem (since $U$-admissible blow-ups are cofinal). Thus we have $\K(X'\on Z') \simeq \K(\Coh_{Z'}^{\leq 1} (X'))$ for every $X'\in\Mdf(X,U)$; in the connective case this is due to Thomason \cite[5.7~(e)]{tt90} and the general case is due to Hiranouchi-Mochizuki \cite[Thm.~3.3]{hiranouchi-mochizuki}\footnote{In \emph{loc.~cit.}\ the category $\Coh_{Z'}^{\leq 1} (X')$ is denoted by $\mathbf{Wt}^1(X'\on Z')$ and it is shown that the category of bounded complexes in this category is equivalent to the derived category $\Perff(X'\on Y')$.}. Hence we have that 
	\begin{align*}
	\K(\skp{X}_U\on\tilde{Z}) 
	&\simeq \colim_{X'\in\Mdf(X,U)} \K(X'\on Z') 
	&\simeq \colim_{X'\in\Mdf(X,U)} \K\bigl(\Coh_{Z'}^{\leq 1} (X') \bigr)
	&\simeq \K\bigl((\Coh_{\tilde{Z}}^{\leq 1}(\skp{X}_U)\bigr).
	\end{align*}
By Theorem~\ref{leq1-resumee--thm} we have an equivalence  
$\K\bigl((\Coh_{\tilde{Z}}^{\leq 1}(\skp{X}_U)\bigr) \To \K\bigl((\Coh_{\tilde{Z}}(\skp{X}_U)\bigr)$
so that we are done by Proposition~\ref{ZR-G-theory-fibre-sequence--prop}.
\end{proof}

\begin{cor} \label{ZR-K=G-U-regular--cor}
Assume that $X$ is divisorial and noetherian and that $U$ is \emph{regular} and dense in $X$. Then the canonical map
	\[
	\K(\skp{X}_U) \To \Ge(\skp{X}_U)
	\]
is an equivalence.
\end{cor}
\begin{proof}
We have a commuative diagram 
\[\begin{xy}\xymatrix{
	\K(\skp{X}_U\on\tilde{Z}) \ar[r] \ar[d] 
	& \K(\skp{X}_U) \ar[r] \ar[d] & \K(U) \ar[d] \\
	\K(\Coh_{\tilde{Z}}(\skp{X}_U)) \ar[r]
	& \Ge(\skp{X}_U) \ar[r] & \Ge(U)
}\end{xy}\]
where the upper line is a fibre sequence by design and the lower line is a fibre sequence by Proposition~\ref{ZR-G-theory-fibre-sequence--prop}. Furthermore, the first vertical map is an equivalence by Theorem~\ref{ZR_K-theory_support_g-theory--thm} and the third vertical map is an equivalence since $U$ is regular.
\end{proof}

\subsection{Homotopy invariance}

\begin{defin}
For an integer $n\geq 1$ we define
	\[
	\skp{X}_U[t_1,\ldots,t_n] := \lim_{X'\in\Mdf(X,U)} X'[t_1,\ldots,t_n].
	\]
\end{defin}

\begin{lemma}
There exists a canonical isomorphism
	\[
	\skp{X}_U[t_1,\ldots,t_n] \To[\cong] \skp{X}_U \times_XX[t_1,\ldots,t_n].
	\]
of locally ringed spaces.
\end{lemma}
\begin{proof}
This follows since different limits commute among each other, namely:
	\begin{align*}
	\skp{X}_U[t_1,\ldots,t_n]
	&= \lim_{X'\in\Mdf(X,U)} X'[t_1,\ldots,t_n] \\
	&= \lim_{X'\in\Mdf(X,U)}\,\lim(X'\to X\ot X[t_1,\ldots,t_n]) \\
	&= \lim\bigl( (\lim_{X'\in\Mdf(X,U)}X') \to X \ot X[t_1,\ldots,t_n]\bigr) \\
	&= \skp{X}_U \times_XX[t_1,\ldots,t_n].
	\end{align*}
\end{proof}

\begin{comment}
\begin{lemma}
There exists a canonical isomorphism 
	\[
	\skp{X[t]}_{U[t]} \To[\cong] \skp{X}_U[t]
	\]
of locally ringed spaces.
\end{lemma}

\begin{proof}
If $p\colon X'\to X$ is a $U$-modification, then the base change $p[t] \colon X'[t]\to X[t]$ is a $U[t]$-modification. Hence the projections $\skp{X[t]}_{U[t]} \to X[t]$ induce the desired morphism.

For every  $U[t]$-modification $Y\to X[t]$ there exists a $U[t]$-admissible blow-up $\Bl_Z(X[t])\to X[t]$ dominating $Y$ \cite[2.1.5]{temkin08}. Set $Z_0 := Z \cap X$ where we consider $X$ as a closed subscheme of $X[t]$ via the zero section. By functoriality for blow-ups we get a morphism $\Bl_{Z_0}(X) \To \Bl_Z(X[t])$ which induces by base change a morphism $\Bl_{Z_0}(X)[t] \To \Bl_Z(X[t])$. As the projection $X[t]\to X$ is flat, we have that $\Bl_{Z_0[t]}(X[t]) = \Bl_{Z_0}(X)[t]$ \cite[13.91 (2)]{gw10}. Thus every $U[t]$-modification is dominated by a base change of a $U$-modification so that the morphism in question is an isomorphism.
\end{proof}
\end{comment}

%%%%%%%%%

\begin{defin}
For a locally ringed space $Y$ we define
	\[
	\Nil(Y) := \fib\bigl( \K(\Nilu(\Vect(Y))) \to \K(Y) \bigr)
	\]
where the map is induced by the forgetful functor $\Nilu(Y) \to \Vect(Y)$ (\cf{} Definition~\ref{nilu--defin}).\footnote{The author apologises for the possibly confusing notation which tries to be coherent with existing literature.}
\end{defin}

\begin{lemma} \label{Nil-Nilu_k-fib-seq--lemma}
Let $Y$ be a quasi-compact and quasi-separated scheme. Then the category $\Nilu(\Coh(Y))$ is an abelian category and $\Nilu(\Vect(Y))$ is an exact subcategory. Furthermore the fibre sequence
	\[
	\Nil(Y) \To \K(\Nilu(\Vect(Y))) \To \K(Y)
	\]
splits so that we get a decomposition
	\[
	\K(\Nilu(\Vect(Y))) \simeq \K(Y) \times \Nil(Y).
	\]
\end{lemma}
\begin{proof}
The first part follows straighforwardly from the fact that $\Vect(Y)$ is an exact subcategory of the abelian category $\Coh(Y)$. The forgetful funtor $\Nilu(\Vect(Y)) \to \Vect(Y)$ is split by the functor sending a vector bundle $E$ to the pair $(E,0)$ so that the claim follows.
\end{proof}

\begin{remark} \label{Nil-fibre-sequence--remark}
There exists a split fibre sequence
	\[
	\Nil(\skp{X}_U) \To \K(\Nilu(\skp{X}_U)) \To \K(\skp{X}_U)
	\]
defined as the colimit of the respective fibre sequeces for every $X'\in\Mdf(X,U)$ coming from Lemma~\ref{Nil-Nilu_k-fib-seq--lemma}. Hence there exists a split fibre sequence
	\[
	\Nil(\skp{X}_U\on\tilde{Z}) \To \K(\Nilu(\skp{X}_U\on\tilde{Z})) \To \K(\skp{X}_U\on\tilde{Z})
	\]
which is defined to be the fibre of the map
	\[\begin{xy}\xymatrix{
	\Nil(\skp{X}_U) \ar[r] \ar[d] 
	& \K(\Nilu(\skp{X}_U)) \ar[r] \ar[d] & \K(\skp{X}_U) \ar[d] \\
	\Nil(U) \ar[r]  & \K(\Nilu(U)) \ar[r] & \K(U) 
}\end{xy}\]
of fibre sequences.
\end{remark}

\begin{thm} \label{K-ZR-support-homotopy-invariance--thm}
Assume that $X$ is divisorial and noetherian and that $U$ is dense in $X$. Denote $\tilde{Z}:=\skp{X}_U\setminus U$. Then the canonical map
	\[
	\K\bigl( \skp{X}_U\on\tilde{Z} \bigr) \To \K\bigl( \skp{X}_U[t]\on\tilde{Z}[t] \bigr)
	\]
is an equivalence.
\end{thm}
\begin{proof}
First note that the map in question identifies with the induced map
	\[
	\K_{\geq 0}\bigl( \skp{X}_U\on\tilde{Z} \bigr) \To \K_{\geq 0}\bigl( \skp{X}_U[t]\on\tilde{Z}[t] \bigr)
	\]
of the respective connective covers due to \cite[Prop~7]{kerz-icm}, \cf{} Theorem~\ref{ZR_negative_K_support_vanishing--thm}; for the target use the relative version with respect to the morphism $X[t]\to X$ which is smooth of finite presentation.

We will show that the cofibre $\NK(\skp{X}_U\on\tilde{Z})$ of the desired equivalence vanishes. For every $n\in\Z$ we have that
	\[
	\Nil_n\bigl( \skp{X}_U\on\tilde{Z} \bigr) \cong \NK_{n+1}\bigl( \skp{X}_U\on\tilde{Z} \bigr)
	\]
by the corresponding statement for schemes \cite[V.8.1]{weibel} since K-theory commutes with filtered colimits of exact categories with exact functors. Thus it suffices to show that $\Nil_n(\skp{X}_U\on\tilde{Z})$ vanishes, \ie{} that the map 
	\[ \tag{$\spadesuit$}
	\K_{\geq 0}\bigl( \Nilu(\skp{X}_U\on\tilde{Z}) \bigr) \To \K_{\geq 0}\bigl( \skp{X}_U\on\tilde{Z} \bigr)
	\]
is an equivalence. In order to see this, consider the following diagram of connective K-theory spectra.
\[ \begin{xy} \xymatrix{
	\K_{\geq 0}\bigl( \Ch^b_{\tilde{Z}}(\Nilu(\skp{X}_U)) \bigr) \ar[d]^\delta \ar[r]^\epsilon
	& \K_{\geq 0}\bigl( \Nilu(\skp{X}_U\on\tilde{Z} \bigr) \ar[r]^-{\spadesuit} \ar[d]^\zeta
	& \K_{\geq 0}\bigl( \skp{X}_U\on\tilde{Z} \bigr) \ar[d]^\alpha
	\\
	\K_{\geq 0}(\bigl( \Ch^b(\Nilu(\Coh_{\tilde{Z}}(\skp{X}_U))) \bigr) 
	& \K_{\geq 0}\bigl( \Nilu(\Coh_{\tilde{Z}}(\skp{X}_U)) \bigr) \ar[r]^-\beta \ar[l]_-\gamma
	& \K_{\geq 0}\bigl( \Coh_{\tilde{Z}}(\skp{X}_U) \bigr) 
} \end{xy} \]
The map $\alpha$ is an equivalence due to Theorem~\ref{ZR_K-theory_support_g-theory--thm}. The map $\beta$ is an equivalence due to the D\'{e}vissage Theorem \cite[V.4.1]{weibel} since  every object $(E,\nu)$ in $\Nilu(\Coh(\tilde{Z}))$ has a filtration
	\[
	(E,\nu) \supset (\ker(\nu),\nu) \supset (\ker(\nu^2),\nu) \supset \ldots \supset 0
	\]
whose quotients have trivial endomorphisms. The map $\gamma$ is an equivalence by the Gillet-Waldhausen Theorem \cite[V.2.2]{weibel}. The map $\delta$ is a composition of equivalences by Lemma \ref{platification-support-lemma}~(iii) and the Thomason-Trobaugh Resolution Theorem \cite[V.3.9]{weibel}, see Lemma~\ref{nil-resolution--lemma} below. The map $\epsilon$ is an equivalence by combining the Waldhausen Localisation Theorem and the Waldhausen Approximation Theorem \cite[V.2.5]{weibel}, see Lemma~\ref{nil-localisation--lemma} below. Thus the map $\zeta$ is an equivalence and we are done.
\end{proof}

\begin{lemma} \label{nil-resolution--lemma}
The canonical maps
\begin{align*}
\K_{\geq 0}\bigl( \Ch^b_{\tilde{Z}}(\Nilu(\skp{X}_U)) \bigr) \To 
\K_{\geq 0}(\bigl( \Ch^b_{\tilde{Z}}(\Nilu(\Modfp(\skp{X}_U))) \bigr) \oT
\K_{\geq 0}(\bigl( \Ch^b(\Nilu(\Coh_{\tilde{Z}}(\skp{X}_U))) \bigr) 
\end{align*}
are equivalences.
\end{lemma}
\begin{proof}
For each of the desired equivalences we want to invoke the Waldhausen Approximation Theorem \cite[V.2.5]{weibel}. In \emph{loc.~cit.}\ we set $\Mcal:=\Nilu(\Modfp(\skp{X}_U))$.

First we treat the right map for which we set $\Bcal:= \Ch^b_{\tilde{Z}}(\Mcal)$ and $\Acal:=\Ch^b_{\tilde{Z}}(\Nilu(\Coh_{\tilde{Z}}(\skp{X}_U)))$ in \emph{loc.~cit.} We claim that every complex in $\Bcal$ is quasi-isomorphic to a complex in $\Acal$. If the complex is concentrated in one degree, then this follows from the equivalence $\Coh_{\tilde{Z}}(\skp{X}_U) \to \Modfp_{\tilde{Z}}(\skp{X}_U)$ of Theorem~\ref{leq1-resumee--thm}. For a general complex $0\to F_m \to\ldots\to F_{n+1} \to[\partial] F_n\to 0$ in $\Acal$ the complexes $0\to F_m \to\ldots\to F_{n+1} \to 0$ and $0\to\coker(\partial)\to 0$ are (by induction) quasi-isomorphic to complexes in $\Acal$ which can be put together so that the initial complex is quasi-isomorphic to a complex in $\Acal$.

For the left equivalence we keep $\Bcal:= \Ch^b_{\tilde{Z}}(\Mcal)$ and reset $\Acal:=\Ch^b_{\tilde{Z}}(\Nilu(\skp{X}_U))$ in \emph{loc.~cit.} For a complex in $\Bcal$ we may assume that it lies in $\Ch^b_{\tilde{Z}}(\Nilu(\Coh_{\tilde{Z}}(\skp{X}_U)))$ by the previous argument. In this case, Lemma~\ref{platification-support-lemma}~(iii) and an induction argument as above yield a quasi-isomorphism to a complex in $\Acal$. 
\end{proof}

\begin{lemma} \label{nil-localisation--lemma}
There is a canonical equivalence
\[
\K_{\geq 0}\bigl( \Ch^b_{\tilde{Z}}(\Nilu(\skp{X}_U)) \bigr) \To  \K_{\geq 0}\bigl( \Nilu(\skp{X}_U\on\tilde{Z}) \bigr).
\]
\end{lemma}
\begin{proof}
By definition we have a fibre sequence
\[
\K_{\geq 0}\bigl( \Nilu(\skp{X}_U\on\tilde{Z}) \bigr) \To \K_{\geq 0}\bigl( \Nilu(\skp{X}_U) \bigr) \To \K_{\geq 0}\bigl( \Nilu(U) \bigr)
\]
so that is suffices to show that the sequence
\[
\K_{\geq 0}\bigl( \Ch^b_{\tilde{Z}}(\Nilu(\skp{X}_U)) \bigr) \To \K_{\geq 0}\bigl( \Nilu(\skp{X}_U) \bigr) \To \K_{\geq 0}\bigl( \Nilu(U) \bigr)
\]
is a fibre sequence as well. Using the Gillet-Waldhausen Theorem \cite[V.2.2]{weibel} we see that it suffices to show that the sequence
\[
\K_{\geq 0}\bigl( \Ch^b_{\tilde{Z}}(\Nilu(\skp{X}_U)) \bigr) \To \K_{\geq 0}\bigl( \Ch^b(\Nilu(\skp{X}_U)) \bigr) \To \K_{\geq 0}\bigl( \Ch^b(\Nilu(U)) \bigr)
\]
is a fibre sequence. The latter follows by combining the Waldhausen Localisation Theorem and the Waldhausen Approximation Theorem \cite[V.2.5]{weibel}. In order to see that the assumptions of the cited theorem are satisfied, note that the assumptions are satisfied without the ``$\Nilu$'', that the operators ``$\Ch^b$'' and ``$\Nilu$'' commute with each other, and that ``$\Nilu$'' preserves the assumptions (as it is basically a diagram category).
\end{proof}

\begin{cor} \label{K-ZR-U-regular-homotopy-invariance--cor}
Assume that $X$ is divisorial and noetherian and that $U$ is \emph{regular} and dense in $X$. Then  the canonical projection $\skp{X}_U[t]\to \skp{X}_U$ induces an equivalence
	\[
	\K\bigl(\skp{X}_U\bigr) \To[\simeq] \K\bigl(\skp{X}_U\bigr[t]\bigr).
	\]
\end{cor}
\begin{proof}
We have a commuative diagram of fibre sequences
\[\begin{xy}\xymatrix{
	\K(\skp{X}_U\on\tilde{Z}) \ar[r] \ar[d] 
	& \K(\skp{X}_U) \ar[r] \ar[d] & \K(U) \ar[d] \\
	\K(\skp{X}_U[t]\on\tilde{Z}[t]) \ar[r]
	& \K(\skp{X}_U[t]) \ar[r] & \K(U[t])
}\end{xy}\]
where the left map is an equivalence by Theorem~\ref{K-ZR-support-homotopy-invariance--thm} and the right map is an equivalence by homotopy invariance of algebraic K-theory for regular noetherian schemes \cite[6.8]{tt90}. 
\end{proof}

\subsection{Vanishing of negative K-theory} \label{schlichting-conjecture--remark}
There was a conjecture of Schlichting whereby the negative K-theory of an abelian category vanishes \cite[9.7]{schlichting06}. Schlichting himself proved the vanishing in degree -1 and subsequently the conjecture in the noetherian case \cite[9.1, 9.3]{schlichting06}. A generalised conjecture about stable \infcats{} with noetherian heart was studied by Antieau-Gepner-Heller \cite{antieau-gepner-heller}. Recently, Neeman \cite{neeman-counterexample} gave a counterexample to both conjectures, Schlichting's and Antieau-Gepner-Heller's. Nevertheless, the following example provides a non-noetherian abelian category whose negative K-theory vanishes.

\begin{example} \label{ZR-not-coherent--ex}
Let $k$ be a discretely valued field with valuation ring $k^\circ$, uniformiser $\pi$, and residue field $\tilde{k}=k^\circ/(\pi)$. Consider the scheme $X_0 := \Spec(k^\circ\skp{t})$, its special fibre $X_0/\pi = \Spec(\tilde{k}\skp{t})$ and the open complement $U:=\Spec(k\skp{t})$. Let $x_0 \colon \Spec(\tilde{k}) \to X_0/\pi$ be the zero section. Then we define the blow-up $X_1 := \Bl_{\{x_0\}}(X_0)$. Its special fibre $X_1/\pi$ is an affine line over $\tilde{k}$ with an $\P^1_{\tilde{k}}$ attached to the origin. We get a commutative diagram
\[
\begin{xy} 
\xymatrix@R=2.5ex@C=1.2ex{ 
	E_1 \ar[rr] \ar[dd] && X_1/\pi \ar[rr] \ar[dd] && X_1  = \Bl_{\{x_0\}}(X_0) \ar[dd] \\
	&\quad\quad\quad\quad&& \Bl_{\{x_0\}}(X_0/\pi) \ar[dl]_\cong \ar[ul] \ar[ur] \\
	\{x_0\} \ar[rr] && X_0/\pi\ar[rr]  && X_0
}
\end{xy}
\]
where the two squares are cartesian, all horizontal maps are closed immersions, the blow-up $\Bl_{\{x_0\}}(X_0/\pi) \to X_0/\pi$ is an isomorphism, and the map $\Bl_{\{x_0\}}(X_0/\pi) \to X_1/\pi$ is also closed immersion. Hence there is a closed immersion $X_0/\pi \inj X_1/\pi$ which splits the canonical projection. Now choose a closed point $x_1 \colon \Spec(\tilde{k}) \to E_1\setminus X_0/\pi$ and define $X_2 := \Bl_{\{x_1\}}(X_1)$ and iterate this construction. Thus we obtain a stricly increasing chain of closed immersions
	\[
	X_0/\pi \inj X_1/\pi \inj \ldots \inj X_n/\pi \inj \ldots\ldots
	\]
and we can consider $X_i/\pi$ as a closed subscheme of $X_n/\pi$ for $i < n$. In each step, $X_{n+1}/\pi$ is obtained from $X_n/\pi$ by attaching a $\P^1_{\tilde{k}}$ to at a closed point not contained in $X_{n-1}/\pi$.
Passing to the Zariski-Riemann space $\skp{X_0}_U$, its special fibre $\skp{X_0}_U/\pi$ admits a strictly decreasing chain of closed subschemes 
	\[
	\skp{X_0}_U/\pi = p_0^*(X_0/\pi) \supsetneq p_1^*(\overline{X_1/\pi\setminus X_0/\pi}) \supsetneq \ldots \supsetneq p_n^*(\overline{X_n/\pi\setminus X_{n-1}/\pi}) \supsetneq \ldots\ldots
	\]
where $p_n \colon \skp{X_0}_U/\pi \to X_n/\pi$ denotes the canonical projection. Thus we obatin a strictly increasing chain of ideals
	\[
	0 = \Ical_0 \subsetneq \Ical_1 \subsetneq \ldots \subsetneq \Ical_n \subsetneq \ldots\ldots
	\]
in $\O_{\skp{X_0}_U/\pi}$ where $\Ical_n$ denotes the ideal sheaf corresponding to $p_n^*(\overline{X_n/\pi\setminus X_{n-1}/\pi})$. By construction, they all have support in $\tilde{Z}=\skp{X_0}_U/\pi$. Thus the category $\Coh_{\skp{X_0}_U/\pi}(\skp{X_0}_U)$ is a non-noetherian abelian category and whose negative K-theory vanishes due to Theorem~\ref{ZR_negative_K_support_vanishing--thm} and Theorem~\ref{ZR_K-theory_support_g-theory--thm}. 
\end{example}

%%%%%%%%%%%%%%%%%%%%%%%%%%%%%%%%%%%%%%%%
\appendix
%%%%%%%%%%%%%%%%%%%%%%%%%%%%%%%%%%%%%%%%

%%%%%%%%%%%%%%%%%%%%%%%%%%%%%%%%%%%%%%%%
	\section{Limits of locally ringed spaces}
	\label{sec-limits-lrs}
%%%%%%%%%%%%%%%%%%%%%%%%%%%%%%%%%%%%%%%		

In this section we collect for the convenience of the reader some facts about locally ringed spaces and filtered limits of those. This is based on the exposition by Fujiwara-Kato \cite[ch.~0, \S 4.2]{fuji-kato}.
	
\begin{defin} \label{cohesive--def}
We say that a locally ringed space $(X,\O_X)$ is \df{cohesive} iff its structure sheaf $\O_X$ is coherent. 
\end{defin}

\begin{example}
For a locally noetherian scheme $X$, an $\O_X$-module is coherent if and only if it is finitely presented \stacks{01XZ}. Thus a locally noetherian scheme is a cohesive locally ringed space.
\end{example}

\begin{lemma}[{\cite[ch.~0, 4.1.8, 4.1.9]{fuji-kato}}] \label{cohesive_coherent_fin-pres--lem}
\begin{enumerate}
	\item If $(X,\O_X)$ is cohesive, then an $\O_X$-module $F$ is coherent if and only if it is finitely presented. 
	\item If $f\colon (X,\O_X) \to (Y,\O_Y)$ is a morphism of locally ringed spaces and $(X,\O_X)$ is cohesive, then $f^* \colon \Mod(Y) \to \Mod(X)$ restricts to a functor $f^* \colon \Coh(Y) \to \Coh(X)$.
\end{enumerate}
\end{lemma}

\begin{defin} \label{sober-coherent--def}
A topological space is said to be \df{coherent} iff it is quasi-compact, quasi-separated, and admits an open basis of quasi-compact subsets. A topological space is called \df{sober} iff it is a T$_0$-space and any irreducible closed subset has a (unique) generic point.
\end{defin}

\begin{example} 
The underlying topological space of a quasi-compact and quasi-separated scheme is coherent and sober.
\end{example}

\begin{prop}[{\cite[ch.~0, 2.2.9, 2.2.10]{fuji-kato}}] \label{limit-top-spaces--prop}
Let $(X_i,(p_{ij})_{j\in I})_{i\in I}$ be a filtered system of topological spaces. Denote by $X$ its limit and by $p_i\colon X\to X_i$ the projection maps.
\begin{enumerate}
	\item Assume that the topologies of the $X_i$ are generated by quasi-compact open subsets and that the transition maps $p_{ij}$ are quasi-compact. Then every quasi-compact open subset $U\subset X$ is the preimage of a quasi-compact open subset $U_i\subset X_i$ for some $i\in I$.
	\item Assume that all the $X_i$ are coherent and sober and that the transition maps $p_{ij}$ are quasi-compact. Then $X$ is coherent and sober and the $p_i$ are quasi-compact.
\end{enumerate} 
\end{prop}

\begin{prop} \label{limit_spaces--prop}
Let $(X_i,\O_{X_i},(p_{ij})_{j\in I})_{i\in I}$ be a filtered system of locally ringed spaces. Then its limit $(X,\O_X)$ in the category of locally ringed spaces exists. Let $p_i\colon X\to X_i$ be the canonical projections. 
\begin{enumerate}
	\item The underlying topological space $X$ is the limit of $(X_i)_{i\in I}$  in $\Top$.
	\item $\O_X = \colim_{i\in I} p_i\inv\O_{X_i}$ in $\Mod(X)$.
	\item For every $x\in X$ have $\O_{X,x} = \colim_{i\in I} \O_{X_i,p_i(x)}$ in $\Ring$.
\end{enumerate}
Assume additionally that every $X_i$ is coherent and sober and that all transitions maps are quasi-compact. 
\begin{enumerate} 
	\item[(iv)] The canonical functor
		\[
		\colim_{i\in I} \,\Mod\fp(X_i) \To \Mod\fp(X) 
		\]
	is an equivalence in the 2-category of categories. In particular, for any finitely presented $\O_X$-module $F$ there exists an $i\in I$ and a finitely presented $\O_{X_i}$-module $F_i$ such that $F\cong p_i^*F_i$.
	\item[(v)] For any morphism $\phi\colon F\to G$ between finitely presented $\O_X$-modules there exists an $i\in I$ and a morphism $\phi_i\colon F_i\to G_i$ between finitely presented $\O_{X_i}$-modules such that $\phi\cong p_i^*\phi_i$. Additionally, if $\phi$ is an isomorphism or an epimorphism, then one can choose $\phi_i$ to be an isomorphism or an epimorphism, respectively.
	\item[(vi)] For every $i\in  I$ let $F_i$ be an $\O_{X_i}$-module and for every $i\leq j$ in $I$ let $\phi_{ij} \colon p_{ij}^*F_i \to F_j$ be a morphism of $\O_{X_j}$-modules such that $\phi_{ik} = \phi_{jk} \circ p_{jk}^*\phi_{ij}$ whenever $i\leq j\leq k$ in $I$. Denote by $F$ the $\O_X$-module $\colim_i p_i^*F_i$. Then the canonical map
		\[
		\colim_{i\in I} \Ho^*(X_i,F_i) \To \Ho^*(X,F)
		\]
is an isomorphism of abelian groups.
\end{enumerate}
\end{prop}
\begin{proof}
The existence, (i), (ii), and (iii) are \cite[ch.~0, 4.1.10]{fuji-kato} and (iv) and (v) are \cite[ch.~0, 4.2.1--4.2.3]{fuji-kato}. Finally, (vi) is \cite[ch.~0, 4.4.1]{fuji-kato}.
\end{proof}

%%%%%%%%%%%%%%%%%%%%%%%%%%%%%%%%%%%%%%%%%%%%%%%%%%%%%%%%%%
\bibliography{0-literature}
\bibliographystyle{amsalpha}
\end{document}

%% file: 0-header.tex
\usepackage{amscd,amsmath,amsthm}

\usepackage{fouriernc} %font of cisinski-deglise:etale motives
\usepackage[mathscr]{euscript}
\usepackage{MnSymbol} % This makes \mathcal as usual because fouriernc changes it

\usepackage{enumerate}
\usepackage{color}
\usepackage[pdfusetitle,unicode]{hyperref}
\usepackage[cmtip,all]{xy} % centering label of arrow via e.g. ar[r]^-{label} (the minus makes it centering)
\usepackage{enumitem}
\usepackage{stmaryrd}					%z.B. für Widerspruchspfeil (\lightning)
\usepackage{datetime}
\usepackage{setspace}					% Zeilenabstand
\usepackage{microtype}					% macht Raender und Grauwert besser
\usepackage{graphicx}					% Graphik
\usepackage{epsfig}						% Graphik
\usepackage{verbatim}
\usepackage[english]{babel}		% neue Rechtschreibung
\usepackage{ngerman}
\usepackage[utf8]{inputenc} 			% Umlaute ermoeglichen
\usepackage{enumitem}					% Aufzaehlungen modifizieren
\usepackage{xspace} 					% setzt Leerzeichen, wenn welche hingehören (vor nächstem wort, aber nicht vor "." oder ")")
\hypersetup{
  linktoc=page
}

\def\resp{{\sfcode`\.1000 resp.}}
\def\ie{{\sfcode`\.1000 i.e.}}
\def\eg{{\sfcode`\.1000 e.g.}}

\def\cf{{\sfcode`\.1000 cf.}}

\setitemize{leftmargin=0.75cm}

%% file: 0-math-header.tex
\newcommand{\mathletter}[1]{\mathbf{#1}}
\newcommand{\N}{\mathletter{N}}				% natural numbers
\newcommand{\Z}{\mathletter{Z}}				% integral numbers
\newcommand{\Pe}{\mathletter{P}}
\renewcommand{\P}{\Pe}

\newcommand{\A}{\mathletter{A}}

\renewcommand{\to}[1][]{\overset{#1}{\rightarrow}}		
\newcommand{\To}[1][]{\overset{#1}{\longrightarrow}}	
\newcommand{\inj}[1][]{\overset{#1}{\hookrightarrow}}		% injective maps
\newcommand{\surj}[1][]{\overset{#1}{\twoheadrightarrow}}		% surjective maps
\newcommand{\infl}[1][]{\overset{#1}{\rightarrowtail}}		% inflation
\newcommand{\ot}[1][]{\overset{#1}{\leftarrow}}	 			% vice-versa map
\newcommand{\oT}[1][]{\overset{#1}{\longleftarrow}} 			% vice-versa map
\newcommand{\inv}{^{-1}}								% inverse
\newcommand{\Colon}{\,\colon\,}
\newcommand{\skp}[1]{\langle #1\rangle}				% inner product, generation brackets

\DeclareMathOperator{\Hom}{Hom}				% Hom
\DeclareMathOperator{\im}{im}				% im
\DeclareMathOperator{\coker}{coker}			% coker
\DeclareMathOperator{\Ho}{H}					% homology
\DeclareMathOperator{\Bl}{Bl}				% Bl
\DeclareMathOperator{\K}{K}						% K-theory
\DeclareMathOperator{\Ge}{G}						% G-theory
\newcommand{\on}{\,\,\mathrm{on}\,\,}			% K-theory with support
\newcommand{\OXU}{\O_{\skp{X}_U}} 				% structure sheaf Zariski-Riemann
\DeclareMathOperator{\NK}{NK}

\newcommand{\df}[1]{\textbf{#1}}					% definiendum in capital letters
\newcommand{\infcats}{$\infty$"~categories}

\renewcommand{\phi}{\varphi}
\renewcommand{\epsilon}{\varepsilon}
\renewcommand{\rho}{\varrho}

\theoremstyle{definition}				% normal
	\newtheorem{defin}{Definition}[section]
	\newtheorem{example}[defin]{Example}
	\newtheorem{remark}[defin]{Remark}
	\newtheorem*{remark*}{Remark}

	\newtheorem{reminder}[defin]{Reminder}
	\newtheorem*{notation*}{Notation}

\theoremstyle{plain} 					% inner text italic
	\newtheorem{thm}[defin]{Theorem}
	\newtheorem*{thm*}{Theorem}
	\newtheorem{prop}[defin]{Proposition}
	\newtheorem{lemma}[defin]{Lemma}
	\newtheorem{cor}[defin]{Corollary}
	\newtheorem*{cor*}{Corollary}

		\newcounter{zaehler}

\theoremstyle{remark} 					% name of environment italic, not bold

\newcounter{formulanumber}[defin]

\newcommand{\vquad}{\vspace{6pt}}

\DeclareMathOperator*{\colim}{colim}

\DeclareMathOperator{\fib}{fib}

\newcommand{\mymathcal}[1]{\mathcal{#1}}

\renewcommand{\O}{\mathcal{O}}		
			
\newcommand{\Acal}{\mathcal{A}}
\newcommand{\Bcal}{\mathcal{B}}

\newcommand{\Ecal}{\mathcal{E}}

\newcommand{\Ical}{\mathcal{I}}

\newcommand{\Mcal}{\mymathcal{M}}

\DeclareMathOperator{\RZ}{RZ}

\newcommand{\fp}{^\mathrm{fp}}

\newcommand{\mathcat}[1]{\mathrm{#1}}

\newcommand{\Ring}{\mathcat{Ring}}
\newcommand{\Top}{\mathcat{Top}}

\newcommand{\Vect}{\mathcat{Vec}}
\newcommand{\Coh}{\mathcat{Coh}}
\newcommand{\Mod}{\mathcat{Mod}}
\newcommand{\Modfp}{\Mod^\mathrm{fp}}
\newcommand{\Modfpleq}{\Mod^\mathrm{fp,\leq 1}}

\newcommand{\Mdf}{\mathcat{Mdf}}

\newcommand{\Ch}{\mathcat{Ch}}

\newcommand{\Perff}{\mathcat{Perf}}
\newcommand{\Nil}{\mathcat{Nil}}
\newcommand{\Nilu}{\underline{\Nil}}

\DeclareMathOperator{\Spec}{Spec}		% Spec
\newcommand{\carre}[4]{\[\begin{xy}\xymatrix{#1\ar[r]\ar[d]&#2\ar[d]\\#3\ar[r]&#4}\end{xy}\]}
\newcommand{\carrelift}[4]{\[\begin{xy}\xymatrix{#1\ar[r]\ar[d]&#2\ar[d]\\#3\ar[r]\ar@{.>}[ur]&#4}\end{xy}\]}

\newcommand{\lift}[3]{\[\begin{xy}\xymatrix{&#1\ar[d]\\#2\ar[r]\ar@{.>}[ur]&#3}\end{xy}\]}
\newcommand{\liftmap}[6]{\[\begin{xy}\xymatrix{&#1\ar[d]^{#2}\\#3\ar@{.>}[ur]^{#4}\ar[r]_{#5}&#6}\end{xy}\]}

\reversemarginpar		% Randnotiz rechts, da in amsart normeilerweise links.
\newcommand{\stacks}[1]{\cite[\href{http://stacks.math.columbia.edu/tag/#1}{Tag #1}]{stacks-project}}